\documentclass[12pt]{amsart}
\usepackage{GGMath}
\DeclareCaptionSubType*[arabic]{figure}
\DeclareCaptionSubType*[arabic]{table}

\renewcommand{\emph}[1]{\textit{#1}}

\newcommand{\lp}{\ensuremath{\oplus}}
\newcommand{\lm}{\ensuremath{\otimes}}

\newcommand{\base}[2]{\ensuremath{\left. #2\right|_{#1}}}

\begin{document}
\title{Maximally additively reducible subsets of integers}
\author{Gal Gross}
\address{Department of Mathematics, University of Toronto,
Toronto, ON, Canada M5S 2E4}
\email{g.gross@mail.utoronto.ca}

\begin{abstract}
Let $A, B \subseteq \nats$ be two finite sets of natural numbers. We say that  $B$ is an \demph{additive divisor} for $A$ if there exists some $C \subseteq \nats$ with $A = B+C$. We prove that among those subsets of $\{0, 1, \ldots, k\}$ which have $0$ as an element, the full interval $\{0, 1, \ldots,k\}$ has the most divisors. To generalize to sets which do not have $0$ as an element, we prove a correspondence between additive divisors and \demph{lunar multiplication}, introduced by Appelgate, LeBrun and Sloane (2011) in their study of a kind of min/max arithmetic. The number of binary lunar divisors is related to compositions of integers which are restricted in that the first part is greater or equal to all other parts. We establish some bounds on such compositions to show that $\{1, \ldots, k\}$ has the most divisors among all subsets of $\{0, 1, \ldots, k\}$. These results resolve two conjectures of LeBrun et al.~regarding the maximal number of lunar binary divisors, a special case of a more general conjecture about lunar divisors in arbitrary bases. We resolve this third conjecture by generalizing from sum-sets to sum-multisets.
\end{abstract}

\maketitle

\section{Introduction}
\label{s:preface}
Let $(G, +)$ be a commutative group, and $A, B \subseteq G$ subsets. The \demph{sumset} $A+B$ (also called \demph{Minkowski sum}) is the set of pairwise sums
\[A+B \eqdef \st{a+b}{a \in A, b \in B}.\]
Classical additive number theory studies direct problems: given a certain set $A$, what can we say about its sumset $A+A$, or iterated sumsets $nA = A+\cdots + A$ (with $n$ summands)? (See \cite{Nathanson1} for an excellent introduction.) In contrast, inverse problems try to extract information about $A$ from information about its sumset. (See \cite{Nathanson2}; and \cite{TaoVu} for an overview of both direct and inverse problems). One such inverse problem is the question, which subsets are sumsets? The asymptotic version of this question was first raised by Ostmann \cite{Ostmann}: we say that a set of positive integers $C \subseteq \nats^+$ is (additively) \demph{reducible} if there are some $A, B \subseteq \nats^+$ each with more than one element, such that $A+B = C$; otherwise, $C$ is said to be \demph{irreducible}. Similarly, we call $C$ \demph{asymptotically reducible} if there are some $A, B \subseteq \nats^+$ each with more than one element, and and integer $m \in \nats$ such that $(A+B) \insect [m, \infty) = C \insect [m, \infty)$; otherwise, $C$ is said to be \demph{asymptotically irreducible}. It is easy to see that the set $P$ of primes is irreducible (because of $2,3 \in P$), and Ostmann conjectured that it is asymptotically irreducible. This cojecture, sometimes referred to as the inverse Goldbach problem, remains unresolved. It has since been placed in the wider context of the ``sum-product phenomenon'' as exemplified by Erd\H{o}s and Szemer\'edi's \cite{erdosszem}. (See also \cite{elsh} for a summary of recent progress.)

Regardless of multiplicative structure, Wirsing \cite{wirsing} has proved that almost all subsets of $\nats$ are asymptotically irreducible, and hence also irreducible. (To interpret ``almost all'' one identifies subsets of $\nats$ with their binary encoding, and thus with the interval $[0, 2]\subseteq \reals$; see \S\ref{s:lunar}.) This paper is concerned with the similar but opposite question: which subsets $C \subseteq [0, N] \subseteq \nats$ are maximally reducible?

\begin{defi}
	Let $C \subseteq \nats$ be a set of natural numbers. We say that $A$ is an \demph{additive divisor} (or \demph{sumset divisor} or simply \demph{divisor}) for $C$, if there exists some $B \subseteq N$ such that $C = A+B$.

	For any finite set of natural numbers $C \subseteq \nats$, we denote $d(C)$ the number of sumset divisors. That is, $d(C)$ is the number of distinct sets $A \subseteq \nats$ such that there exists some $B \subseteq \nats$ with $C = A+B$.
\end{defi}

The trivial decomposition $C = C + \{0\}$ shows that every set has at least two divisors. Additively irreducible sets have exactly two. Fix some $k \in \nats$ and consider all subsets of $[0, k]$. In \S\ref{s:0-divisors} we develop a correspondence (called ``$k$-promotion'') between divisors of subsets which have $0$ as an element. Theorem \ref{crlodd} of \S\ref{s:0-divisors} shows that among those subsets which have $0$ as an element, the full interval $\{0, 1, \ldots,k\}$ has the most divisors. 

In \S\ref{s:lunar} we assign each finite subset of $\nats$ a binary number, and prove that the sumset operation corresponds to lunar multiplication on binary numbers; a multiplication operator defined by Applegate, LeBrun, and Sloane \cite{lunar} in their study of alternative systems of arithmetic on digits, in which long addition and long multiplication can be performed without ``carries''. This new correspondence connects the number of sumset divisors to the number of lunar divisors. In their paper, LeBrun et al. establish a correspondence between the number of binary lunar divisors of $m$ and the number of restricted compositions of $\ell$, where $\ell$ is the number of $1$'s in the binary representation of $m$. The compositions are restricted in that the first part is greater or equal to all other parts. In \S\ref{s:counting} we prove that the table enumerating such restricted compositions can be easily constructed using properties of the forward difference from the finite calculus. We then establish some bounds on such restricted compositions. We use these bounds to prove Theorem \ref{crleven} of \S\ref{s:nonnormal}, which shows that $\{1, \ldots, k\}$ has the most divisors among all subsets of $[0, k]$ (that is, we remove the restriction that $0$ is an element).

Due to the above correspondence between sumsets and lunar numbers, Theorems \ref{crlodd} and \ref{crleven} resolve Conjectures 13 and 14 of LeBrun et al.~These conjectures are a binary version of a more general Conjecture 12, since lunar arithmetic is defined for arbitrary bases. This conjecture is resolved in \S\ref{s:bases} (see Theorem \ref{thm:bases}), in which we prove a correspondence between base $b$ lunar numbers and arrays of subsets of $\nats$ of height $(b-1)$. We conclude in \S\ref{s:conclusion} with some open questions.
\section{Divisors of 0-rooted sets}
\label{s:0-divisors}

Throughout the paper we denote by $[k]$ the full interval $[k] = [0, k]\insect \nats = \{0, \ldots, k\}$.

\begin{defi}
	We say $A \subseteq \nats$ is a \demph{$0$-rooted set} if $\min{A} = 0$. For $k \in \nats$, let $\mcal{Z}_k$ denote the collection of $0$-rooted sets whose maximal element is $k$:
	\[\mcal{Z}_{k} = \st{A \subseteq \nats}{\min(A) = 0\,,\, \max(A) = k}.\]
\end{defi}
For convenience we also introduce $\mcal{Z}_{\leq k} = \Union_{\ell \leq k}\mcal{Z}_\ell$, and $\mcal{Z} = \Union_{k \in \nats} \mcal{Z}_k$ the collection of finite $0$-rooted sets. 

The purpose of this section is to prove that among $0$-rooted sets $\mcal{Z}_{\leq k}$, the full interval $\{0, \ldots, k\}$ has the most number of divisors. Given any $A \in \mcal{Z}_{\leq k}$, we start by describing a procedure for turning factors of $A$ into factors of $[k]$ in a process we call \demph{$k$-promotion}. 

\begin{defi}
	Let $A \in \mcal{Z}_{\leq k}$ and suppose that $A = B+C$ with $\max(B) \leq \max(C)$. We define the set $C_B$ as follows:
	\[C_B = C \union (([k]\setminus A) \insect [\max(B)-1]) \union ((([k]\setminus A)-\{\max(B)\}) \insect \nats)\]
	That is, for each $s \in [k]\setminus A$: if $s<\max(B)$ we append $s$ to $C$; while if $s\geq \max(B)$ we append $s-\max(B)$ to $C$.
\end{defi}

	\begin{lemma}
		Let $A \in \mcal{Z}_{\leq k}$, and suppose that $A = B+C$. Then $B+C_B = [k]$.
	\end{lemma}
	\begin{proof}
		Since $C \subseteq C_B$ we have $A = B+C \subseteq B+C_B$. Moreover, by construction, $[k]\setminus A \subseteq B+C_B$. Thus, $[k] = A \union ([k]\setminus A) \subseteq B+C_B$. On the other hand, $\max(B+C_B) = \max(B)+\max(C_B)$. Now, $\max(C_B) \leq \max\{k-\max(B), \max(C)\}$ so that $\max(B)+\max(C_B) \leq \max\{k, \max(B)+\max(C)\} = \max\{k, \max(A)\} \leq k$. Moreover, $\min(B+C_B) = 0$. Thus, $B+C_B \subseteq [k]$.
	\end{proof}

\begin{figure}[!h]
\begin{adjustbox}{width=\columnwidth,center}
\begin{tabular}{lclclclclcl}
$\{0, 3\}$ & $+$ & $\{0, 4\}$ & $=$ & $\{0, 3, 4, 7\}$ &$\leadsto$&
$\{0, 3\}$ &$+$ & $\{0, 1, 2, 3, 4\}$ & $=$ & $[7]$ \\

$\{0, 3\}$ & $+$ & $\{0, 4\}$ & $=$ & $\{0, 3, 4, 7\}$ &$\leadsto$&
$\{0, 3\}$ &$+$ & $\{0, 1, 2, 3, 4, 5\}$ & $=$ & $[8]$ \\

$\{0, 3\}$ & $+$ & $\{0, 1, 3\}$ & $=$ & $\{0, 1, 3, 4, 6\}$ &$\leadsto$&
$\{0, 3\}$ &$+$ & $\{0, 1, 2, 3\}$ & $=$ & $[6]$ \\

$\{0, 3\}$ & $+$ & $\{0, 1, 3\}$ & $=$ & $\{0, 1, 3, 4, 6\}$ &$\leadsto$&
$\{0, 2, 3\}$ &$+$ & $\{0, 1, 3\}$ & $=$ & $[6]$ \\
\end{tabular}
\end{adjustbox}
\caption{\small Example of $k$-promotion}
\end{figure}
 
Now, each factor $B$ of $A$ appears in one or more factorizations. We may apply the procedure above to each such factorization. We let $F(B)$ denote the resulting set of factors of $[k]$. That is, for each $C\subseteq A$ such that $B+C = A$: if $\max(B) \leq \max(C)$ we let $B \in F(B)$; if $\max(B) \geq \max(C)$ we let $B_C \in F(B)$ (where $B_C$ is given by the procedure described above). (Note that this means that if there is some $C$ with $\max(C) = \max(B)$, then both $B, B_C \in F(B)$.)

\begin{figure}[h]
{\small
\begin{align*}
\{0, 2, 3, 4, 5, 6\} &=
\{0, 2, 3\} + \{0, 2, 3\} &&F(\{0, 2, 3\}) = \{\{0, 2, 3\};\,\{0, 1, 2, 3\}\}\\
&= \{0, 2\} + \{0, 3, 4\} &&F(\{0, 2\}) = \{\{0, 2\}\}\\
&=\{0, 2\} + \{0, 2, 3, 4\}
\end{align*}
}%
\caption{\small Example of $F(B)$ for $A = \{0, 2, 3, 4, 5, 6\}$}
\end{figure}

\begin{thm}
\label{thm:FB}
	If $B, D$ are different divisors of $A\in \mcal{Z}_{\leq k}$, then $F(B) \insect F(D) = \emptyset$. 
\end{thm}
	
\begin{proof}
	First note that for $A = [k]$ and any divisor $B$ of $A$ we have $F(B) = \{B\}$ so the claim follows trivially. Assume therefore that $A \subsetneq [k]$. We start by showing that $B \in F(B) \implies B \notin F(D)$.

	Suppose that $B \in F(B)$. Then there exists some $C$ with $\max(C) \geq \max(B)$ and $B+C = A$. We already know that $B \neq D$, and all other elements of $F(D)$ are of the form $D_E$. Assume therefore that there exists some $E$ with $\max(E)\leq \max(D)$ and $D+E = A$. If $D \not\subseteq B$, then we are done since $D \subseteq D_E$. Suppose therefore that $D \subseteq B$ so in particular $\max(D) \leq \max(B)$. We therefore have the chain
	\[\max(E) \leq \max(D) \leq \max(B) \leq \max(C)\]
	On the other hand, $\max(A) = \max(E) + \max(D) = \max(B)+\max(C)$. Therefore, $\max(B) = \max(D)$ and so $\max(E) = \max(C)$, and we have a chain of equalities
	\[\max(E) = \max(D) = \max(B) = \max(C).\]
	Let us denote this common element by $m$, and let $s \in [k]\setminus A \neq \emptyset$. There are two options:
	\begin{itemize}
		\item $s < m$, in which case $s \in D_E$. Since $s \notin A \supseteq B$, this shows that $B \neq D_E$.
		\item $s \geq m$, in which case $s-m \in D_E$. Assume for contradiction that $s-m \in B$. Since $m \in C$ we would have $s = (s-m)+m \in B+C = A$, which is a contradiction. Thus, $s-m \notin B$ and $B \neq D_E$.
	\end{itemize}

	We conclude that $B \in F(B) \implies B\notin F(D)$.

	All other elements of $F(B)$ are of the form $B_C$. So we will now show that $B_C \in F(B) \implies B_C \notin F(D)$. First note that by the argument above, $D \in F(D) \implies D \notin F(B)$, so that $D \in F(D) \implies B_C \neq D$. All other elements of $F(D)$ are of the form $D_E$.

	We are therefore assuming that there exist some $C, E$ such that
		\begin{align*}
			&A = B+C\stp \mbox{ and }\stp \max(B) \geq \max(C), \\
			&A = D+E\stp \mbox{ and }\stp \max(D) \geq \max(E).
		\end{align*}

	Assume for contradiction $B_C = D_E$. 

	First suppose that $B \subseteq D$. Since $B \neq D$, there must exist some $d \in D$ such that $d \notin B$. Since $d = d+0 \in D+E = A$ we have $d \notin [k]\setminus A$. However, $d \in D \subseteq D_E = B_C$, so that there must be some $s \in [k]\setminus A$ with $d = s-\max(C)$. Now, $s = d+\max(C) \notin A$ implies $\max(C) \notin E$. Thus, $\max(E) \neq \max(C)$. However, $[k] = B_C+C = D_E+E$ implies $k = \max(B_C)+\max(C) = \max(D_E)+\max(E)$. Since $B_C = D_E$ we have $\max(B_C) = \max(D_E)$ so $\max(C) = \max(E)$, which is a contradiction.

	On the other hand, $B\not\subseteq D$ implies there is some $b \in B$ such that $b\notin D$. Analogous argument to the one above then shows that $\max(E) \neq \max(C)$, which again contradicts the assumption $B_C = D_E$.

\end{proof}

Theorem \ref{thm:FB} is enough to establish the maximality of $d([k])$ among sets of $\mathcal{Z}_{\leq k}$. The following two claims will help to show that it is also the unique maximum.

\begin{lemma}
\label{claim:0oddunique}
	Let $3 \leq k \in \nats$ be odd, and suppose that $A \in \mcal{Z}_{\leq k}$ with $A \subsetneq [k]$. Then $F_k = \{0, (k+1)/2\}$ is a factor of $[k]$ which does not arise from $k$-promotion. That is, $F_k \notin F(B)$ for any factor $B$ of $A$.
\end{lemma}

\begin{proof}
	It is clear that $F_k$ is a factor of $[k]$, since
	\[[k] = [(k-1)/2]+\{0, (k+1)/2\}.\]
	Assume for contradiction that $F_k \in F(B)$ for some factor $B$ of $A$. Then either:
	\begin{itemize}
		\item $B = F_k$ and $B \in F(B)$. That is, there exists some $C$ for which $C+F_k = A$ and $\max(C) \geq \max(F_k)$. But $\max(F_k) = (k+1)/2$ and $(k+1)/2+(k+1)/2 = k+1 > k = \max(A)$. This is a contradiction.
		\item $F_k = B_C$ for some $C$ and $B_C \in F(B)$. That is, there exists some $C$ with $\max(B) \geq \max(C)$ and $B+C = A$, and $B_C = F_k$. Since $B \subseteq B_C = F_k$ we must have $B \subseteq \{0, (k+1)/2\}$. If $B = \{0\}$, the assumption $\max(B) \geq \max(C)$ implies $C = \{0\}$, in which case $B_C = [k] \neq F_k$ (since $k \geq 3$). Thus, $B = \{0, (k+1)/2\}$.

		However, from $C+B_C = [k]$ we obtain $C+F_k = [k]$. In particular, $[(k-1)/2] \subseteq C$ in which case $[k] \subseteq B+C = A$, contradicting the assumption $A \subsetneq [k]$.
	\end{itemize}
\end{proof}

In contrast to the odd case, it is straightforward to verify that all factors of $[4]$, for example, arise from $4$-promotion. We must weaken the hypothesis in the previous lemma from an absolute statement to a relative one:

\begin{lemma}
\label{claim:even0unique}
	Let $4 \leq k \in \nats$ be \emph{even}. Then for any $A \in \mcal{Z}_{\leq k}$ with $A \subsetneq [k]$, there exists some factor $F_k$ of $[k]$, such that $F_k \notin F(B)$ for any factor $B$ of $A$.
\end{lemma}

\begin{proof}
	For $A = [k]\setminus \{2\}$ we use $F_k = \{0, 2\}$. This is indeed a factor of $[k]$, since $[k] = \{0, 2\} + [k-2]$ (for example). On the other hand, for any factor $B$ of $A$ we have $2 \notin B$. Thus, $F_k \in F(B)$ if and only if $B_C = F_k$ for some $C$. That is, there exists some $C$ with $\max(B) \geq \max(C)$ and $B+C = A$, and $B_C = F_k$. Since $B \subseteq B_C = F_k = \{0, 2\}$ and $2 \notin B$ we must have $B = \{0\}$. Then $\max(B) \geq \max(C)$ implies $C = \{0\}$. Then $A = \{0\}$ and $B_C = [k] \neq \{0, 2\}$. Thus, $F_k \notin F(B)$ for any factor $B$ of $A$.

	If $A \neq [k]\setminus \{2\}$ we use $F_k = \{0, 1, 3, 5, \ldots, k-1\}$. First,
	\[[k] = \{0, 1\} + \{0, 1, 3, 5, \ldots, k-1\}\]
	shows that $F_k$ is indeed a factor of $[k]$. Assume for contradiction that $F_k \in F(B)$ for some factor $B$ of $A$. Then either:
	\begin{itemize}
		\item $B = F_k$ and $B \in F(B)$. That is, there exists some $C$ for which $C+F_k = A$ and $\max(C) \geq \max(F_k) = k-1$. This is clearly impossible since $k \geq \max(A) = \max(B)+\max(C) \geq 2(k-1) = 2k-2 \implies k \leq 2$, contradicting our assumption that $k \geq 4$.

		\item $F_k = B_C$ for some $C$ and $B_C \in F(B)$. That is, there exists some $C$ with $\max(B) \geq \max(C)$ and $B+C = A$, and $B_C = F_k$. We have $C+B_C = [k]$, that is $C+F_k = [k]$. Since $\max({F_k}) = (k-1)$ this implies $\max(C) \leq 1$. Thus $C = \{0, 1\}$.

		Now, $B \subseteq B_C =  F_k$. Since $A \subsetneq [k] = C+F_k$ we must have $B \subsetneq F_k$. We will show that the only element of $F_k$ missing from $B$ is $1$. Let $2x+1 \in B_C\setminus B$. Then,
		\begin{align*}
			A &= B+C \subseteq \{0, 1, 3, \ldots, 2x-1, 2x+3, \ldots, k-1\} + \{0, 1\}\\
			 &= [k]\setminus\{2x+1, 2x+2\}.
		\end{align*}
		Thus, $\{2x+1, 2x+2\} \subseteq [k]\setminus A$. In particular, $2x = (2x+1)-\max(C) \in B_C$. Since the only even element in $B_C$ is $0$, we must have $x=0$. Thus, the only element of $B_C$ missing from $B$ is $1$. In other words, $B = \{0, 3, 5, \ldots, k-1\}$. Then, $A = B+C = [k]\setminus \{2\}$, contradicting our assumption that $A \neq [k]\setminus \{2\}$.
	\end{itemize}
\end{proof}

These two lemmas and Theorem \ref{thm:FB} together imply:

\begin{thm}
\label{crlodd}
	The set $[k]$ is the unique maximum of $d(\cdot)$ in $\mcal{Z}_{\leq k}$. 
\end{thm}

\begin{proof}
	Given some $0$-rooted set $A \subsetneq [k]$ we have a map $B \mapsto F(B)$ taking each factor $B$ of $A$, to a set of factors of $[k]$. Theorem \ref{thm:FB} implies
	\[d([k]) \geq \sum_{B \mbox{ \scriptsize divides } A}\crd{F(B)} \geq \sum_{B \mbox{ \scriptsize divides } A} 1 = d(A).\]
	Thus, $[k]$ is a maximum of $d(\cdot)$ in $\mathcal{Z}_{\leq k}$. It is easy to see by direct computation that $[k]$ is the unique maximum for $k=0, 1, 2$ (with $1, 2, 3$ factors respectively). Lemma \ref{claim:0oddunique} and Lemma \ref{claim:even0unique} show that it is also the unique maximum for $k \geq 3$.
\end{proof}
\section{Lunar Arithmetic}
\label{s:lunar}
Applegate, LeBrun, and Sloane \cite{lunar} define base-$b$ lunar\footnote{Originally published under the name `dismal arithmetic', the authors have come to prefer `lunar arithmetic' instead. See \cite{brady} and the relevant OEIS entries \cite{sloane}.} addition $\lp$ as the digitwise $\max$ operation. To multiply two digits we take their $\min$. Long multiplication is similar to regular multiplication, except that lunar addition implies there are no carries. The figures below reproduce the examples from their article:

\begin{figure}[!h]
    \begin{subfigure}[b]{0.5\linewidth}
        \centering
        \begin{tabular}{lr}
		& $169$ \\
		$\lp$ & $248$ \\ \hline
		& $269$
		\end{tabular}
        \caption{base 10 lunar addition}
        \label{fig:10lunarplus}
    \end{subfigure}\hfill
    \begin{subfigure}[b]{0.5\linewidth}
    	\centering
      \begin{tabular}{lr}
		& $169$ \\
		$\lm$ & $248$ \\ \hline
		& $168$ \\
		$\lp$ & $144$\stp \\
		$\lp$ & $122$\stp\stp \\ \hline
		&$12468$
	   \end{tabular}
        \caption{base 10 lunar multiplication}
        \label{fig:10lunartimes}
    \end{subfigure}
    \caption{\small Lunar arithmetic}\label{fig:lunararithmetic}
\end{figure}

Le Brun et al.~then show that $\lp$ and $\lm$ are commutative and associative, and $\lm$ distributes over $\lp$. They proceed to study analogues of number-theoretic constructions ``including primes, number of divisors, sum of divisors, and the partition function.'' \cite{lunar} 

In particular, they define $d_b(n)$ as the number of lunar divisors of $n$ in base $b$. Section 6 of their paper contains a series of conjectures about the properties of $d_b(n)$, which we reproduce below for ease of reference. Note that $\base{b}{a_ma_{m-1}\ldots a_1}$ denotes a base-$b$ representation. Following are Conjectures 12-14 in \cite{lunar}.

\setcounter{conj}{11}

\begin{conj}[LeBrun et al.]
\label{conjbases}
	In any base $b \geq 3$, among all $k$-digit numbers $n$, $d_b(n)$ has a unique maximum at $n = (b^k-1)/(b-1) = 111\ldots 1|_b$.
\end{conj}

\begin{conj}[LeBrun et al.]
\label{conjeven}
	In base $2$, among all $k$-digit numbers $n$, the maximal value of $d_2(n)$ occurs at $n = 2^k-2 = \base{2}{111\ldots 10}$, and this is the unique maximum for $n \neq 2, 4$.
\end{conj}

\begin{conj}[Part I; LeBrun et al.]
\label{conjodd}
	In base 2, among all \emph{odd} $k$-digit numbers $n$, $d_2(n)$ has a unique maximum at $n = 2^k-1 = \base{2}{111\ldots 111}$.
\end{conj}

\setcounter{conj}{13}

\begin{conj}[Part II; LeBrun et al.]
\label{conjodd2}
	In base 2, among all \emph{odd} $k$-digit numbers $n$, if $k \geq 3$ and $k \neq 5$, the second-largest value of $d_2(n)$ occurs at $n = 2^k-3 = \base{2}{111\ldots 101}$, and possibly other values of $n$.
\end{conj}

The sequence $d_2(\base{2}{1\ldots 1})$ in particular appears to count many phenomena, which are documented in \cite{sloane} (see also \cite{bijection} for explicit bijections between some of the interpretations). Le Brun et al.~count $d_2(\base{2}{1\ldots 1})$ by exhibiting a generating function for the sequence; based on an argument originally due to Richard Schroeppel.

We now prove that the sumset operation corresponds to base $2$ lunar multiplication, facilitating the study of divisibility properties.

Let $\mcal{F}$ denote the collection of finite subsets of $\nats$, and let $\mcal{B}$ denote the set of finite binary sequences. There is a natural bijection $\beta: \mcal{F} \to \mcal{B}$. First, $\beta(\emptyset) = 0$. Next, for any nonempty $A \in \mcal{F}$ we define the binary number $\beta(A) = c_{\max(A)}\ldots c_1c_0|_{2}$ as follows: for $0 \leq i \leq \max(A)$:
\[c_i = \begin{cases}
	1 & \mbox{ if } i \in A,\\
	0 & \mbox{ if } i \notin A.
\end{cases}\]

The key observation is
\begin{equation}
\label{key}
	\beta(A+B) = \beta(A)\lm\beta(B).
\end{equation}

This is easy to see when viewing $A+B$ as $\Union_{b \in B} A+ \{b\}$. Thus, for each element $b$ of $B$, we are shifting every element of $A$ by $b$.

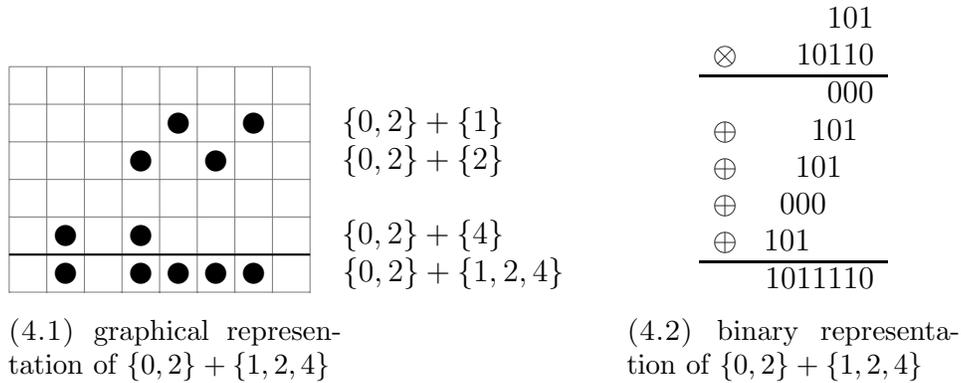
\begin{figure}[!h]
    \begin{subfigure}[b]{0.35\linewidth}
        \centering
        	\begin{tikzpicture}
				\draw[step=0.5cm,color=gray] (-5,1.99) grid (-1,5);
				\node at (-1.75, +4.25) {$\CIRCLE$};
				\node at (-2.75, +4.25) {$\CIRCLE$};
				\node at (+0.5, +4.25) {$\{0, 2\} + \{1\}$};
				\node at (-2.25, +3.75) {$\CIRCLE$};
				\node at (-3.25, +3.75) {$\CIRCLE$};
				\node at (+0.5, +3.75) {$\{0, 2\}+\{2\}$};
				\node at (-3.25, +2.75) {$\CIRCLE$};
				\node at (-4.25, +2.75) {$\CIRCLE$};
				\node at (+0.5, +2.75) {$\{0, 2\}+\{4\}$};
				\draw[thick] (-5,+2.5) -- (-1,+2.5);
				\node at (-1.75, +2.25) {$\CIRCLE$};
				\node at (-2.25, +2.25) {$\CIRCLE$};
				\node at (-2.75, +2.25) {$\CIRCLE$};
				\node at (-3.25, +2.25) {$\CIRCLE$};
				\node at (-4.25, +2.25) {$\CIRCLE$};
				\node at (+0.9, +2.25) {$\{0, 2\} + \{1, 2, 4\}$};
			\end{tikzpicture}
        \caption{graphical representation of $\{0, 2\}+\{1, 2, 4\}$}
        \label{fig:graphsumsets}
    \end{subfigure}\hfill
    \begin{subfigure}[b]{0.35\linewidth}
        \centering
    	\begin{tabular}{lr}
		& $101$ \\
		$\lm$ & $10110$ \\ \hline
		& $000$ \\
		$\lp$ & $101$\stp \\
		$\lp$ & $101$\stp\stp \\
		$\lp$ & $000$\stp\stp\stp\\
		$\lp$ & $101$\stp\stp\stp\stp \\ \hline
		&$1011110$
		\end{tabular} 
        \caption{binary representation of $\{0, 2\}+\{1, 2, 4\}$}
        \label{fig:binsumsets}
    \end{subfigure}
    \caption{\small Two representations of sumsets}\label{fig:sumsetrep}
\end{figure}

\begin{thm}
	$\beta: \mcal{F} \to \mcal{B}$ is a monoid-homomorphism, where $\mcal{F}$ is equipped with the sumset operation, and $\mcal{B}$ with the lunar multiplication operation.
\end{thm}

\begin{proof}
	The fact that $(\mcal{F}, +)$ is a commutative monoid follows from the fact that $(\nats, +)$ is. The fact that $(\mcal{B}, \lm)$ is a monoid is proved in Le Brun et al.~In fact, they show that for any base $m\geq 2$, if $\mcal{A} = \{0, 1, 2, \ldots, b-1\}$ the set of polynomials (the standard representation for a positional counting system) equipped with lunar addition and lunar multiplication $(\mcal{A}[X], \lp, \lm)$ is a commutative semiring.

	We have $\beta(\{0\}) = 1$, so that the neutral elements are mapped to each other. Let $A, B \in \mcal{F}$. We need to prove formula \eqref{key}. The result is clear when one of $A, B$ is $\emptyset$. Assume therefore that $A, B \neq \emptyset$. Identifying the corresponding binary numbers $\beta(A) = a_{\max(A)}\ldots a_1a_0|_{2}$ and $\beta(B) = b_{\max(B)}\ldots b_1b_0|_{2}$ with the polynomials $P = \sum_{i=0}^{\max(A)}a_i2^i$ and $Q = \sum_{i=0}^{\max(B)}b_i2^i$ we have
	\[\beta(A) \lm \beta(B) = \sum_{i=0}^{\max(A)+\max(B)}c_i2^i\]
	where
	\[c_i = \bigoplus_{j+k=i}a_j\lm b_k = \max\st{\min\{a_j, b_k\}}{j+k=i}.\]
	Thus, $c_i = 1$ if and only if there exist $j, k \in \nats$ with $j+k=i$ such that $a_j=1$ and $b_k = 1$. On the other hand, by the definition of $\beta$, the $i$-th digit of $\beta(A+B)$ is $1$ if and only if $i \in A+B$. That is, if and only if there exist $j, k \in \nats$ such that $j \in A$ and $k \in B$ and $j+k = i$. Again, the definition of $\beta$ implies that these two conditions are the same so that $\beta(A+B) = \beta(A)\otimes\beta(B)$.
\end{proof}

This homomorphism identifies $0$-rooted sets and odd binary numbers so that Conjecture \ref{conjodd} (Part I) is an immediate corollary of Theorem \ref{crlodd} above:

\begin{crl}
	In base 2, among all \emph{odd} $k$-digit numbers $n$, $d_2(n)$ has a unique maximum at $n = 2^k-1 = \base{2}{111\ldots 111}$.
\end{crl}
\section{Counting $d([k])$}
\label{s:counting}
In Section \ref{s:nonnormal} we find the maximum of $d(\cdot)$ among all subsets of $[k]$, not just the $0$-rooted ones.  One important part of the proof is the observation (already made by LeBrun et al.~in \cite{lunar}) that $d([k]\setminus \{0\}) = 2d([k-1])$. We then show that $2d([k-1])>d([k])$. The purpose of this section is to help us establish this inequality by highlighting the connection between the sequence $d([k])$ and Fibonacci numbers of higher-order.

Recall that a \demph{composition} of a natural number $n \in \nats$ is an ordered tuple of \emph{positive} natural numbers $(c_1, c_2, \ldots, c_k)$ such that $n = \sum_{i=1}^k c_i$. If the length of the tuple is $k$, it is called a $k$-composition. Each of the $c_i$ is called a \demph{part} of the composition. It is an easy exercise to show that the total number of compositions of $n$ is $2^{n-1}$. Placing different restrictions on such compositions leads to a rich theory. For example, one may restrict the length of a composition, the size of the parts, the type of the parts, the arrangement of the parts, etc. In particular, integer \demph{partitions} are integer compositions arranged in a non-decreasing order. Other types of restriction have to do with so-called ``statistics'': If $(c_1, c_2, \ldots, c_k)$ is a composition of $n$, we say that a \demph{rise} occurs at position $i$ if $c_{i+1}>c_i$, a \demph{fall} is defined analogously, and a \demph{level} occurs when $c_{i+1}=c_i$. \demph{Statistics} in the context of composition have to do with the number of such rises and falls. Another type of restriction demands that certain patterns be avoided. MacMahon \cite{Mac} was amongst the firsts to study such questions in detail, and we refer the reader to the recent book of Heubach and Mansour \cite{heubach} for an excellent survey of the field.

LeBrun conjectured that $d_2(\base{2}{111\ldots 111})$ (with $k$ 1's) counts the number of compositions of $k$ with the added restriction that the first part is greater or equal to all other parts. For example, $d([3]) = d_2(\base{2}{1111})=5$ corresponding to the five restricted compositions $(4)$; $(3,1)$; $(2,2)$; $(2, 1, 1)$; $(1, 1, 1, 1)$. For ease of reference, we shall call such restricted compositions \demph{headstrong compositions}.

This conjecture was proved by Richard Schroeppel (\cite{lunar}, Theorem 16) and again by Frosini and Rinaldi in \cite{bijection}. For the sake of completion, we reproduce it below in the language of sumsets. 

\begin{thm}[Schroeppel, 2001]
	For any $n \in \nats$, the number $d_2([n])$ equals the number of headstrong compositions of $n+1$.
\end{thm}

\begin{proof} 
	Suppose $n+1 = c_1+c_2+\cdots+c_k$ is a headstrong composition. Then we let
	\begin{align*}
		A &\eqdef \{(n+1)-c_1, (n+1)-(c_1+c_2), (n+1)-(c_1+c_2+c_3), \ldots, 0\},\\
		B &\eqdef [c_1-1]
	\end{align*}
	It is clear that $A+B = [n]$. Conversely, suppose $A+B = [n]$, then also $A+[\max(B)] = [n]$. Suppose $A = \{0=a_0, a_1, \ldots, a_k=\max(A)\}$ with $0 <a_1<\cdots <a_k$. Then,
	\[n+1 = (\max(B)+1) + (a_k-a_{k-1})+(a_{k-1}-a_{k-2}) + \cdots + (a_2-a_1)+(a_1-a_0)\]
	is a headstrong composition of $n+1$, and applying the procedure above to this composition will give us $A$. To see this is indeed a headstrong composition, note that $a+{j+1}-1 \in [n] = A+B$ and let $a_{j+1}-1 = a+b$ for some $a \in A$ and $b \in B$. Since $a_j\leq a_{j+1}-1$ we have $b' = a_{j+1}-1-a_j \geq 0$. On the other hand, $a< a_{j+1}$ implies $a\leq a_j$ so that $b' \leq b \leq \max(B)$. Since $a_j+b'+1 = a_{j+1}$ we must have $a_j+\max(B)+1 \geq a_{j+1}$, that is $\max(B)+1 \geq a_{j+1}-a_j$.
\end{proof}

The bijection in the proof in facts shows that the following corollary holds. The next section studies this further, as it will play an important role in the proof of Conjecture \ref{conjbases}.

\begin{crl}
\label{crl:parts}
	The number of divisors of $[n]$ whose cardinality is exactly $m$ equals the number of headstrong compositions of $n+1$ with exactly $m$ parts.
\end{crl}

Headstrong compositions were first studied by Knopfmacher and Robbins \cite{knopf} who derived generating functions and asymptotics for them. The (ordinary) generating function is given by
\[\sum_{\ell=1}^{\infty}\frac{(1-z)z^\ell}{1-2z+z^{\ell+1}}.\]
The index $\ell$ corresponds to the leading term of the composition. By comparing generating functions it is easy to see that the number of headstrong composition of $n$ starting with $k$ is given by $F(k, n)$, the $n$-th element of the \demph{generalized Fibonacci sequence} $F(k, \cdot)$ defined by the recurrence relation

\begin{align*}
	F(n, k) = \begin{cases}
		0 & \mbox{ if } 1\leq k<n,\\
		1 & \mbox{ if } k = n,\\
		\sum_{j=1}^n F(n, k-j) & \mbox{ if } k> n.
	\end{cases}
\end{align*}

It is traditional to start enumerating the Fibonacci sequence at $0$ (that is, $F_0=0$, $F_1 = 1$ etc). However, we break from this tradition and start our count at $1$ (that is, $F_1 = 0$, $F_2 = 1$), which gives our formula in the proposition below a pleasing symmetry. Thus, $F(2, \cdot)$ are the usual Fibonacci numbers, $F(3, \cdot)$ are the so-called Tribonacci numbers, etc. Note that the sequence $F(1, \cdot)$ is simply the constant sequence $1, 1, 1\ldots$

\begin{table}[!h]
\begin{adjustbox}{width=\columnwidth,center}
\begin{tabular}{l|cccccccccc}
$F(n, k)$& $k=1$ & $k=2$ & $k=3$ & $k=4$ & $k=5$ & $k=6$ & $k=7$ & $k=8$ & $k=9$ & $k=10$ \\ \hline
$n=1$ & $1$ & $1$ & $1$ & $1$ & $1$ & $1$ & $1$ & $1$ & $1$ & $1$ \\
$n=2$ & $0$ & $1$ & $1$ & $2$ & $3$ & $5$ & $8$ & $13$ & $21$ & $34$ \\
$n=3$ & $0$ & $0$ & $1$ & $1$ & $2$ & $4$ & $7$ & $13$ & $24$ & $44$ \\
$n=4$ & $0$ & $0$ & $0$ & $1$ & $1$ & $2$ & $4$ & $8$ & $15$ & $29$\\
$n=5$ & $0$ & $0$ & $0$ & $0$ & $1$ & $1$ & $2$ & $4$ & $8$ & $16$\\
\end{tabular}
\end{adjustbox}
\caption{\small Table of $F(n, k)$}
\end{table}

To get $d_2(\base{2}{111\ldots 111})$ (with $k$ 1's) one simply sums the $k$-th column in the table above. (For a generating-functions-free proof, observe that $F(n,\cdot)$ enumerates the number of compositions with leading term $n$ by the following reasoning. Each composition of $m$ with leading term $n$, can be obtained by appending $1$ to the end of a composition of $m-1$ (with leading term $n$); or by appending $2$ to the end of a composition of $m-2$ (with leading term $n$); and so on until we append $n$ to the end of a composition of $m-n$. We are not double-counting since the compositions differ in their last term; and we can prove inductively that all compositions of $m$ (with leading term $n$) are obtained in that way.) Thus, the number of headstrong compositions of $k$ is given by $\sum_{n=1}^{k}F(n,k) = \sum_{n\geq 1}F(n, k)$.

As mentioned at the beginning of this section, this characterisation simplifies the proof that $2d([k-1])>d([k])$, as a consequence of the following lemma.

\begin{lemma}
\label{fibdouble}
	For $k \geq n$ we have $2F(n, k)\geq F(n, k+1)$, and the inequality is strict for $k\geq 2n$.
\end{lemma}

\begin{proof}
	 Since $F(n, n) = F(n, n+1) = 1$, the claim is clearly true for $k=n$. Since the claim holds trivially for $n=1$, we may assume that $n \geq 2$. For $k>n$ we have from the definition of $F(n, k)$ 
	\begin{align*}
		F(n, k+1) &= \sum_{j=1}^{n}F(n, k+1-j)\\
		&=F(n, k)+\sum_{j=1}^{n-1}F(n, k-j)\\
		&=2F(n, k)-F(n, k-n) \leq 2F(n, k)
	\end{align*}
	with strict inequality for $k \geq 2n$.
\end{proof}

Lemma \ref{fibdouble} gives a lower bound on $F(n, k)$ in terms of the next element of the sequence. The next lemma gives an upper bound showing that
\[\frac{1}{2}F(n, k+1)\leq F(n, k) \leq \frac{2}{3} F(n, k+1).\]
This will be used to show that $[k]\setminus \{0\}$ is $d(\cdot)$-maximal among all $\pset{[k]}\setminus \mcal{Z}_{\leq k}$. These bounds are not asymptotically tight: the reader may recall that $F(2, k)/F(2, k+1) \to 1/\phi$ where $\phi = (1+\sqrt{5})/2 \approx 1.618$. By considering the characteristic equation of the linear recursion, one finds that $F(n, k)/F(n, k+1) \to 1/r$, where $r$ is the largest real root of $x^n(2-x)=1$ (see \cite{mathworld}), though we shall not use this result. 

\begin{lemma}
\label{fibbound}
	For any $n>1$ and $k > n$ we have $3F(n, k) \leq 2F(n, k+1)$. Equality holds if and only if $n=2$ and $k=4$.
\end{lemma}

\begin{proof}
	We have $F(2, 3) = 1$, $F(2, 4) = 2$, $F(2, 5) = 3$, which proves the claim for $n=2$ and $k \leq 4$.

	If $n=2$ and $k > 4$, then $k-1\geq 2n$ so Lemma \ref{fibdouble} shows that
	\begin{align*}
		2F(n, k+1) = 2F(n, k)+2F(n, k-1) > 3F(n, k).
	\end{align*}

	On the other hand, if $n\geq 3$ and $k=n+1$ we have
	\begin{align*}
		2F(n, n+2) = 4 > 3 = 3F(n, n+1)
	\end{align*}
	while if $k\geq n+2$ we have by Lemma \ref{fibdouble}
	\begin{align*}
		2F(n, k+1) &\geq 2F(n, k) + 2F(n, k-1) + 2F(n, k-2)\\
		 &\geq 3F(n, k)+2F(n, n)\\
		 &> 3F(n, k).
	\end{align*}
\end{proof}

\section{Divisors of non-$0$-rooted sets}
\label{s:nonnormal}
Most of the groundwork for Theorem \ref{crleven} is now done. The remaining key observation is Lemma 15 of \cite{lunar}. For completeness, Lemma \ref{L15} proves the relevant part\footnote{Lemma 15 applies to arbitrary bases: ``If the base $b$ expansion of $n$ ends with exactly $r\geq 0$ zeroes, so that $n=mb^r$ with $b\!\not|m$, then $d_b(n) = (r+1)d_b(m)$.''} in the language of sumsets.

For convenience, we introduce the notation $[k+] = [k]\setminus \{0\} = \{1, 2, 3, \ldots, k\}$. For $1\leq k \in \nats$ we also set $\mathcal{Z}^+_{k}$ the collection of sets of positive natural numbers whose maximal element is $k$:
\[\mcal{Z}^+_k \eqdef \st{A \subseteq \nats}{\min(A)>0,\,\max(A) = k}.\]
Finally, we have $\mcal{Z}^+_{\leq k} = \Union_{0<\ell\leq k}\mcal{Z}^+_\ell$. Note that we have a partition of subsets of $[k]$: $\pset{[k]} = \{\emptyset\} \sqcup \mcal{Z}_{\leq k} \sqcup \mcal{Z}^+_{\leq k}$ (disjoint union).

\begin{lemma}[LeBrun et al.]
\label{L15}
	Let $A$ be a finite subset of $\nats$, and let $r \eqdef \min(A)$. Then,
	\[d(A) = (r+1)d(A-\{r\}).\]
\end{lemma}

\begin{proof}
	Note that $A-\{r\} \in \mcal{Z}$ is a $0$-rooted set. Suppose $B+C = A-\{r\}$. Then, for any $0\leq k \leq r$
	\[(B+\{k\})+(C+\{r-k\}) = (B+C)+\{r\} = (A-\{r\})+\{r\} = A\]
	shows that each of $B+\{k\}$ is a divisor of $A$. These are clearly all distinct, so $d(A)\geq (r+1)d(A-\{r\})$.

	Conversely, suppose $B+C = A$. Denoting $b \eqdef \min(B)$ and $c \eqdef \min(C)$, we have $b+c = \min(A) = r$. Thus,
	\[(B-\{b\})+(C-\{c\}) = (B+C)-\{r\} = A - \{r\}.\]
	This shows that $F \mapsto F-\min(F)$ maps divisors of $A$ to divisors of $A-\{r\}$. Since $0 \leq \min(F)\leq r$ (and $F \neq F'$ implies $F-\{k\} \neq F'-\{k\}$), each divisor of $A-\{r\}$ is the image of at most $(r+1)$ divisors of $A$. That is, $d(A) \leq (r+1)d(A-\{r\})$.
\end{proof}

As an immediate consequence of Lemma \ref{L15} we have $d([k+]) = 2d([k-1])$. The following Theorem then implies that $[k+]$ has more divisors than any element in $\mcal{Z}_{\leq k}$.

\begin{thm}
\label{propdoubling}
	$2d([k-1])\geq d([k])$, and the inequality is strict for $k>1$.
\end{thm}

\begin{proof}
	Note that $d([0]) = 1$ and $d([1])=2$, which verifies the claim for $k=1$. Assume therefore that $k\geq 2$. We have seen in Section \ref{s:counting} that
	\[d([k-1]) = \sum_{n=1}^{k}F(n, k).\]
	By Lemma \ref{fibdouble} we have $2F(n, k)\geq F(n,k+1)$ so that for $k>1$ we find
	\begin{align*}
		2d([k-1])-d([k]) &= \sum_{n=1}^{k}(2F(n, k)-F(n, k+1)) - F(k+1, k+1)
	\end{align*}
	Now, $F(1,\cdot)$ is the constant $1$ sequence, so $2F(1, k)-F(1, k+1) = 1$. Moreover, $F(k+1, k+1) = 1$ by definition. Thus,
	\begin{align*}
		2d([k-1])-d([k]) = \sum_{n=2}^{k}(2F(n, k)-F(n, k+1)).
	\end{align*}
	Since $k\geq 2$ the sum is nonempty, and Lemma \ref{fibdouble} shows that each term in the sum is nonnegative. In fact the sum includes the term $2F(k, k)-F(k, k+1) = 2-1$ so it is positive.
\end{proof}

To show that $[k+]$ is the maximum of $d(\cdot)$ in $\mcal{Z}^+_{\leq k}$, we use Lemma \ref{fibbound}.

\begin{thm}
\label{uppereven}
	For any $1\leq j\leq k$ we have $(j+1)d([k-j]) \leq 2d([k-1])$. Equality holds if and only if $j=1$, or $k=3$ and $j=2$.
\end{thm}

\begin{proof}
	Since $d([0]), d([1]), d([2]) = 1, 2, 3$ respectively, it is easy to verify that the claim holds for $k\leq 3$. In particular, for $k=3$ and $j=2$ we have 
	\[3d([1]) = 3\cdot 2 = 2\cdot 3 = 2d([2]).\]
	Assume therefore that $k \geq 4$. The claim trivially holds with equality for $j=1$. We shall prove the strict inequality part of the claim by induction, with base case $j=2$. Recall that $F(1, \cdot) = F(n, n) = F(n, n+1) = 1$ (for any positive $n$). We have by Lemma \ref{fibbound}
	\begin{align*}
		3d([k-2]) &= 3\sum_{n=1}^{k-1}F(n, k-1) \\
		&= 2\cdot 3 + \sum_{n=2}^{k-2}3F(n, k-1)\\
		&< 2(F(1, k)+F(k-1, k)+F(k, k))+\sum_{n=2}^{k-2}2F(n, k)\\
		&= 2\sum_{n=1}^{k}F(n, k)
		= 2d([k-1]).
	\end{align*}
	(Note that the strict inequality is justified since the sum contains at least one element different from $F(2, 4)$.)

	Suppose that for some $j\geq 2$ we know that for all $k\geq \min(4, j)$ we have $(j+1)d([k-j])< 2d([k-1])$. Note that $3d([n-2]) \leq 2d([n-1])$ for any $n \geq 2$ (it is only when we require the inequality to be strict that we need $n \geq 4$). Thus, if $j+1\leq k$ we have
	\begin{align*}
		(j+2)d([k-j-1]) &= (j-1)d([k-j-1])+3d([k-j-1])\\
		&\leq (j-1)d([k-j])+2d([k-j])\\
		&= (j+1)d([k-j])\\
		&< 2d([k-1]).
	\end{align*}
\end{proof}

\mbox{}

\begin{thm}
\label{crlZ+}
	For any $\emptyset\neq A \subsetneq [k+]$ we have $d([k+])\geq d(A)$, and the inequality is strict for $k \neq 3$.
\end{thm}

\begin{proof}
	Let $a \eqdef \min(A)$. Note that $A \subsetneq [k+]$ implies $a \geq 1$. By Proposition \ref{L15} we have $d(A) = (a+1)d(A-\{a\})$. By Corollary \ref{crlodd} we have $d(A-\{a\}) \leq d([\max(A)-a])$ and the inequality is strict if $A-\{a\} \neq [\max(A)-a]$. Moreover, $d([\max(A)-a]) \leq d([k-a])$ and the inequality is strict if $\max(A) \neq k$.

	Therefore, by Theorem \ref{uppereven} we have
	\[d(A) \leq (a+1)d([k-a]) \leq 2d([k-1]) = d([k+])\]
	and the inequality is strict if $k\neq 3$ and $a >1$.

	In summary we have $d(A) \leq d([k+])$ and if $k \neq 3$ equality may only hold if $a=1$ and $\max(A) = k$ and $A-\{1\} = (A-\{a\}) = [\max(A)-a] = [k-1]$. This contradicts the assumption that $A \subsetneq [k+]$. Thus, for $k \neq 3$ the inequality is strict.
\end{proof}

It is also true that for $k=3$ the inequality can fail to be strict. For example, $d([3+]) = 6 = d(\{2, 3\})$. We now have Conjecture \ref{conjeven}.

\begin{thm}
\label{crleven}
	For $k \geq 1$, the set $[k+]$ is the maximum of $d(\cdot)$ in $\pset{[k]}\setminus\{\emptyset\}$, and this is the unique maximum for $k\neq 1, 3$. Equivalently, in base $2$, among all $k$-digit numbers $n$, the maximal value of $d_2(n)$ occurs at $n = 2^k-2 = \base{2}{111\ldots 10}$, and this is the unique maximum for $n \neq 2, 4$.
\end{thm}

\begin{proof}
	By Lemma \ref{L15} we have $d([k+]) = 2d([k-1])$. By Theorem \ref{propdoubling} we have $2d([k-1]) \geq d([k])$ with strictly inequality for $k>1$. By Theorem \ref{crlodd} we know that $[k]$ is the unique maximum of $d(\cdot)$ in $\mathcal{Z}_{\leq k}$. We conclude that $d([k+]) \geq d(A)$ for any $A \in \mathcal{Z}_{\leq k}$, and the inequality is strict for $k>1$.

	Next, Theorem \ref{crlZ+} then shows that for any $A \in \mcal{Z}^+_{\leq k}$ we have $d([k+]) \geq d(A)$, and the inequality is strict if $k \neq 3$ (and $A \neq [k+])$.

	Since $\pset{[k]} = \{\emptyset\} \dunion \mcal{Z}_{\leq k} \dunion \mcal{Z}^+_{\leq k}$ we conclude that $[k+]$ is the maximum of $d(\cdot)$ in $\pset{[k]}\setminus\{\emptyset\}$, and this is the unique maximum for $k\neq 1, 3$.
\end{proof}
\section{The triangle of headstrong compositions}
\label{s:headstrongtriangle}
In Section \ref{s:bases} we prove Conjecture \ref{conjbases}. One important part of the proof is the observation that $d_b$ can be given in terms of powers of $d_2$-divisors (see Theorem \ref{manytoone}). This will require us to compare headstrong compositions by the number of parts (rather than simply count the total number). The purpose of this section is to help us establish a convenient recurrence for these numbers.

Since $0$ is not allowed as a part of a composition, a composition of $n$ may have at most $n$ parts. Letting the rows indicate $n$, and the columns the number of parts, we obtain a triangle of compositions. In the case of unrestricted compositions it is easy to see that this is in fact the Pascal triangle of binomial coefficients. This signifies the importance of compositions in probability. Another example, crucial for this section, is the following Montmort-Moivre\footnote{A similar formulation with a deck of cards is sometimes referred to as Simon Newcomb type problems, popularized in \cite{Mac}.} type problem: consider $m$ urns, each containing $s$ balls labelled $1, \ldots, s$. If one ball is drawn uniformly at random from each of the $m$ urns, what is the probability that the sum of the labels is $n$?

This answer leads to a definition of $C(n, m, s)$, the number of $m$-compositions of $n$ such that no part exceeds $s$. This quantity has been studied by statisticians (see \cite{jordan}, \cite{charl}; and \cite{gencomp} for generalizations where each part is bounded above and below), and we have the generating function
 \begin{align*}
 g(z) &= (z+z^2+\cdots+z^s)(z+z^2+\cdots+z^s)\cdots(z+z^2+\cdots+z^s)\\
 &=z^{m}\frac{(1-z^s)^m}{(1-z)^m}.
 \end{align*}

We are interested in the triangle of headstrong compositions (Table \ref{table:headstrong}), enumerated in sequence A184957 of the OEIS \cite{sloane} (by rows). We let $H(n, k)$ denote the number of a headstrong $k$-compositions of $n$. Note that $\sum_{k=1}^{n}H(n, k) = \sum_{k=1}^{n}F(k, n)$ (though the rows and columns of the two tables do not agree).

\begin{table}[!h]
\begin{adjustbox}{width=\columnwidth,center}
\begin{tabular}{l|ccccccccccc}
$H(n, k)$& $k=1$ & $k=2$ & $k=3$ & $k=4$ & $k=5$ & $k=6$ & $k=7$ & $k=8$ & $k=9$ & $k=10$ \\ \hline
$n=1$ & $1$ &  &  &  &  &  &  &  &  &  \\
$n=2$ & $1$ & $1$ &  &  &  &  &  &  &  &  \\
$n=3$ & $1$ & $1$ & $1$ &  &  &  &  &  &  & \\
$n=4$ & $1$ & $2$ & $1$ & $1$ &  &  &  &  &  & \\
$n=5$ & $1$ & $2$ & $3$ & $1$ & $1$ &  &  &  &  & \\
$n=6$ & $1$ & $3$ & $4$ & $4$ & $1$ & $1$ &  &  &  & \\
$n=7$ & $1$ & $3$ & $6$ & $7$ & $5$ & $1$ & $1$ &  &  &\\
$n=8$ & $1$ & $4$ & $8$ & $11$ & $11$ & $6$ & $1$ & $1$ &  & \\
$n=9$ & $1$ & $4$ & $11$ & $17$ & $19$ & $16$ & $7$ & $1$ & $1$ & \\
$n=10$ & $1$ & $5$ & $13$ & $26$ & $32$ & $31$ & $22$ & $8$ & $1$ & $1$\\
\end{tabular}
\end{adjustbox}
\caption{\small Table of $H(n, k)$, headstrong $k$-compositions of $n$}
\label{table:headstrong}
\end{table}

To introduce the recurrence we first recall that given a function $f: \nats \to \reals$, its \demph{(first forward) difference} $\Delta f$ is defined by $\Delta f(n) = f(n+1)-f(n)$. One then defines recursively $\Delta^k f = \Delta(\Delta^{k-1} f)$. (We refer the interested reader to \cite{concrete} for an introduction, and \cite{jordan} for an extensive treatment of the finite calculus.) Writing the sequence $\Delta^k f$ in row $k$ we obtain the \demph{difference table} for the sequence represented by $f$ (with the convention that $\Delta^0 f = f$). It is conventional to align the table so that the difference of two items appears in between them, and thereby obtain a difference triangle (of course, it is only by curtailing the sequence that the shape of a triangle is obtained).

\begin{table}[!h]
\centering
\begin{tabular}{r|cccccccccccccccc}
$f$ & $1$ & & $5$ & & $14$ & & $30$ & & $55$ \\
$\Delta^1 f$ & & $4$ & & $9$ & & $16$ & & $25$\\
$\Delta^2 f$ & & & $5$ & & $7$ & & $9$ \\
$\Delta^3 f$ & & & & $2$ & & $2$ \\
$\Delta^4 f$ & & & & & $0$ 
\end{tabular}
\caption{\small Difference table for the sum of squares $f(n) = \sum_{k=1}^{n} k^2$}
\end{table}

It can be shown by induction that if $f$ is a polynomial of degree $k$, then $\Delta^{k+1}f = 0$. Conversely, if $\Delta^k f = 0$, then $f$ is a $(k-1)$-degree polynomial. It is clear that in such cases the entire difference table can be recovered from the first diagonal. In particular, if the first diagonal is $(d_0, d_1, \ldots, d_{n-1}, 0)$ the sequence itself is given by $f(n) = \sum_{k=0}^{n-1}d_{k}{n-1\choose k}$. 

The purpose of the current section is to prove that the triangle of headstrong compositions is self-generating in the following manner. The first column and the last element of each row are both $1$, trivially. The $(k+1)$-th diagonal is $H(k+1, 1), H(k+2, 2), H(k+3, 3), \ldots$ If we construct the difference table for this sequence, the first diagonal of the difference table will be exactly the $k$-th row of the headstrong triangle. Thus, starting from the data that the first column and the last element are each $1$ we have the table in Figure \ref{init}. We know that the second diagonal is given by the first row according to the formula $\sum_{k=0}^{0}{n-1\choose k}$, so it is the constant $1$ sequence. This gives us the table in Figure \ref{init2}, and we continue in this fashion. For example, the fifth diagonal is given by 
\[{n-1\choose 0} + 2{n-1\choose 1}+{n-1\choose 2}+{n-1\choose 3} = \frac{1}{6}(n^3-3n^2+14n-6)\]
and so on. It is clear that the whole triangle can be recovered in this manner.

\mbox{}

\begin{table}[!h]
    \begin{subfigure}[b]{0.35\linewidth}
        \centering
        \begin{adjustbox}{width=1.5\columnwidth}
        \begin{tabular}{cccccc}\small
			$1$ &  &  &  &  & \\
			$1$ & $1$ &  &  &  &  \\
			$1$ & $H(3, 2)$ & $1$ &  &  & \\
			$1$ & $H(4, 2)$ & $H(4,3)$ & $1$ &  & \\
			$1$ & $H(5, 2)$ & $H(5,3)$ & $H(5,4)$ & $1$ &  \\
			$1$ & $H(6,2)$ & $H(6,3)$ & $H(6,4)$ & $H(6,5)$ & $1$\\
		\end{tabular}
		\end{adjustbox}
        \caption{initial data}
        \label{init}
    \end{subfigure}\hfill
    \begin{subfigure}[b]{0.35\linewidth}
        \centering
        \begin{adjustbox}{width=1.3\columnwidth}
        \begin{tabular}{cccccc}\small
			$1$ &  &  &  &  &   \\
			$1$ & $1$ &  &  &  &   \\
			$1$ & $1$ & $1$ &  &  &   \\
			$1$ & $H(4, 2)$ & $1$ & $1$ &  &   \\
			$1$ & $H(5, 2)$ & $H(5,3)$ & $1$ & $1$ &  \\
			$1$ & $H(6,2)$ & $H(6,3)$ & $H(6,4)$ & $1$ & $1$ \\
		\end{tabular}
		\end{adjustbox}
        \caption{filled $2$nd diagonal}
        \label{init2}
    \end{subfigure}

        \begin{subfigure}[b]{0.35\linewidth}
        \centering
        \begin{tabular}{cccccc}\small
			$1$ &  &  &  &  & \\
			$1$ & $1$ &  &  &  &  \\
			$1$ & $1$ & $1$ &  &  & \\
			$1$ & $2$ & $1$ & $1$ &  & \\
			$1$ & $H(5, 2)$ & $3$ & $1$ & $1$ &  \\
			$1$ & $H(6,2)$ & $H(6,3)$ & $4$ & $1$ & $1$\\
		\end{tabular}
        \caption{filled $3$rd diagonal}
        \label{init3}
    \end{subfigure}\hfill
    \begin{subfigure}[b]{0.35\linewidth}
        \centering
        \begin{tabular}{cccccc}\small
			$1$ &  &  &  &  & \\
			$1$ & $1$ &  &  &  &  \\
			$1$ & $1$ & $1$ &  &  & \\
			$1$ & $2$ & $1$ & $1$ &  & \\
			$1$ & $2$ & $3$ & $1$ & $1$ &  \\
			$1$ & $3$ & $4$ & $4$ & $1$ & $1$\\
		\end{tabular}
        \caption{filled table}
        \label{init4}
    \end{subfigure}
    \caption{\small Lunar arithmetic}
    \label{table:headstrongrecurrence}
\end{table}

Entry $H(n, k)$ of the headstrong triangle is the $k$-th entry of the $d\eqdef n-k+1$ diagonal, which is goverened by the $d-1$ row. Thus, the claim expresses a recurrence relation of the following form:

\begin{thm}
\label{prop:headstrong}
Let $m, n$ be positive integers. We have
\[H(n, m) = \begin{cases}
	0 & \mbox{ if } m>n,\\
	1 & \mbox{ if } m=n,\\
	1 & \mbox{ if } m=1.
\end{cases}\]
In all other cases $n>m>1$ we have
	\begin{equation}
	\tag{$\star$}
	\label{rec}
	H(n, m) = \sum_{j=1}^{n-m}H(n-m, j){m-1 \choose j-1}
	\end{equation}
\end{thm}

\begin{proof}
	It is easy to see that for $m>1$ we have $H(n, m) = \sum_{s=1}^{n-1}C(n-s, m-1, s)$, which (using the generating function for $C(n, m, s)$ from the beginning of this section) gives us a generating function
 \begin{align*}
 	f(z) &= \sum_{s=1}z^sz^{m-1}\frac{(1-z^s)^{m-1}}{(1-z)^{m-1}}\\
 	\\
 	&= \sum_{s=1}z^{m+s-1}\left(\frac{1-z^s}{1-z}\right)^{m-1}
 \end{align*}
 and observe that this generating function gives the correct result for $m=1$ as well. Thus, it is indeed the generating function for $H(n, m)$.

 On the other hand, the hypothesis for $n>m>1$
 \begin{align*}
 H(n, m) &= \sum_{j=1}H(n-m, j){m-1 \choose j-1} \\
 &=1+\sum_{j=2}H(n-m,j){m-1\choose j-1}\\
 &= 1+ \sum_{j=2}{m-1\choose j-1}\sum_{s=1}C(n-m-s, j-1, s)
 \end{align*}
 gives the generating function
 \[h(z) = z^m+\sum_{s=1}z^{m+s}+\sum_{j=2}{m-1\choose j-1}\sum_{s=1}z^{m+s}z^{j-1}\frac{(1-z^s)^{j-1}}{(1-z)^{j-1}}
 \]
 (Note that the coefficient of $z^m$ accounts for the case $n=m$; while $\sum_{s=1}z^{m+s}$ accounts for the cases where $n>m$ and $m=1$.) However, this is the same generating function for $H(n, m)$, since

\begin{align*}
z^m+\sum_{s=1}z^{m+s}\sum_{j=1}{m-1\choose j-1}\left(z\frac{1-z^s}{1-z}\right)^{j-1}
 		&=z^m+\sum_{s=1}z^{m+s}\left(1+z\frac{1-z^s}{1-z}\right)^{m-1}\\
 		\\
 		&=\sum_{s=1}z^{m+s-1}\left(\frac{1-z^{s}}{1-z}\right)^{m-1}
 \end{align*}
 \end{proof}

 Using Theorem \ref{prop:headstrong} we may now prove a relation between the rows of the triangle of headstrong compositions that will play a key role in our proof of Conjecture \ref{conjbases}.

 \begin{crl}
 \label{crl:headstrongbounds}
 	Let $b \geq 2 \in \nats$. Then, for any positive integer $n \in \nats$,
 	\[2\sum_{m=1}^n H(n,m)b^{m} < \sum_{m=1}^{n+1}H(n+1,m)b^{m}.\]
 \end{crl}

 \begin{proof}
 	It is easy to verify the claim for $n=1$ and $n=2$ directly. For $n=1$ it reduces to $2b < b+b^2$, and for $n=2$ it reduces to $2b+2b^2< b+b^2+b^3$. Assume therefore $n>2$. For $n>m>0$ we have by Theorem \ref{prop:headstrong}
 	\begin{align*}
 		H(n+1, m+1) &= \sum_{j=1}H(n-m, j){m \choose j-1}\\
 		&=\sum_{j=1}H(n-m, j){m-1\choose j-1} + \sum_{j=2}H(n-m, j){m-1\choose j-2}\\
 		&=H(n, m) + \sum_{j=2}H(n-m, j){m-1\choose j-2}.
 	\end{align*}
 	(Note that the last summand is $0$ unless $n-m>1$, reflecting the fact that $H(n+1,n) = H(n,n-1)$.) Denoting $h(n, m) = \sum_{j=2}H(n-m, j){m-1\choose j-2}$, we find
 	\begin{align*}
 		\sum_{m=1}^{n+1} H(n+1,m)b^{m} &= b+\sum_{m=1}^{n} H(n+1,m+1)b^{m+1}\\
 		&=b+\sum_{m=1}^{n} (H(n,m)+h(n,m))b^{m+1}\\
 		&= b+\sum_{m=1}^{n}bH(n,m)b^m + \sum_{m=1}^{n}bh(n,m)b^m\\
 	\end{align*}
 	However, $b\geq 2$ (and all summands are nonnegative) so that
 	\begin{align*}
 		\sum_{m=1}^{n+1} H(n+1,m)b^{m} &= b+\sum_{m=1}^{n}bH(n,m)b^m + \sum_{m=1}^{n}bh(n,m)b^m\\
 		&\geq b+\sum_{m=1}^{n}bH(n,m)b^m \\
 		&\geq b+\sum_{m=1}^{n}2H(n,m)b^m\\
 		&>2\sum_{m=1}^{n}H(n,m)b^m.
 	\end{align*}
 \end{proof}
\section{Sumsets arrays}
\label{s:bases}
We have seen that sumsets correspond to binary lunar multiplication and can be used to analyse lunar divisors. Lunar arithmetic is defined for arbitrary bases $b \geq 2$. We now prove that higher bases correspond to multisets of sumsets. Recall that multisets are ``sets with repetitions.'' While $\{1, 1, 2\}$ and $\{1, 2\}$ represent the same set, they represent two different multisets. A set of natural numbers can be identified with a function $f: \nats \to \{0, 1\}$ which decides set-memebership; i.e.~$n$ belongs to the set if and only if $f(n) = 1$. A multiset of natural numbers can be identified with a function $f: \nats \to \nats$, which decides set-membership and multiplicity. All multisets in this section are finite multisets of natural numbers.

There is a grading of multisets by multiplicity. For $b \in \nats$, let $\mcal{M}^{b}$ denote the collection of finite multisets (of natural numbers) whose maximal multiplicity does not exceed $b$. That is,
\[\mcal{M}^b \eqdef \st{f: \nats \to \nats}{f(\nats) \subseteq [b],\exists k \in \nats.(j>k)\implies f(j)=0}.\] 
Note that $\mcal{M}^1$ is simply the collection of finite subsets of natural numbers, while $\mcal{M}^0$ is the emptyset. We have $\mcal{M}^0 \subsetneq \mcal{M}^1 \subsetneq \mcal{M}^2 \subsetneq \cdots$

We make the following definitions analogously \S\ref{s:lunar}. For $k \in \nats$, let $\mcal{M}_k$ denote the collection of multisets of natural numbers whose maximal element is $k$:
\[\mcal{M}_{k} = \st{f:\nats \to \nats}{f(k) \neq 0,\mbox{ and } j>k \implies f(j)=0}.\]
For convenience we also introduce $\mcal{M}_{\leq k} = \Union_{\ell \leq k}\mcal{M}_\ell$, and $\mcal{M} = \Union_{k \in \nats} \mcal{M}_k$ the collection of finite multisets of natural numbers. Finally, we may combine superscripts and subscripts; so that $\mcal{M}^b_k$ is the collection of multisets of natural numbers whose maximal element is $k$, and such that the multiplicity of any element does not exceed $b$ (such a multiset extends a function $f:[k] \to [b]$ to a multiset $f:\nats \to \nats$ by $j>k \implies f(j) = 0$). Note that $\mcal{M}^b = \bigsqcup_{k \in \nats}\mcal{M}^b_k$ (disjoint union).

We have a choice for defining multi-sumsets. The naive definition simply treats multisets as sets, for example $\{1, 1, 2\}+ \{2\} = \{3, 3, 4\}$. This does not take advantage of the extra-structure of multisets. The definition below takes into account multiplicity, and allows different interactions between ``multiplicity levels'', so that $\{1, 1, 2\}+\{2\} = \{3, 4\}$ while $\{1, 1, 2\}+\{2, 2\} = \{3, 3, 4\}$.

\begin{figure}[!h]
    \begin{subfigure}[b]{0.3\linewidth}
        \centering
       \[
        \begin{mx}
        	\{0,1,2\} \\ \{0,1\} \\ \{0,1\} \\ \{0,1\} \\ \{0,1\} \\
        	\{0,1\} \\ \{0\} \\ \{0\} \\ \{0\} 
        \end{mx}
        +
        \begin{mx}
        	\{0,1,2\} \\ \{0,1,2\} \\ \{0,1\} \\ \{0,1\} \\ \{0\} \\
        	\{0\} \\ \{0\} \\ \{0\} \\ \emptyset 
        \end{mx}
        =
        \begin{mx}
        	\{0, 1, 2, 3, 4\} \\ \{0,1, 2, 3\} \\ \{0,1,2\} \\ \{0,1,2\} \\ \{0,1\} \\
        	\{0,1\} \\ \{0\} \\ \{0\} \\ \emptyset 
        \end{mx}
        \]
        \caption{multiset addition in $\mcal{M}^9$}
    \end{subfigure}\hfill
    \begin{subfigure}[b]{0.3\linewidth}
        \centering
   	 	\begin{tabular}{lr}
		& $169$ \\
		$\lm$ & $248$ \\ \hline
		& $168$ \\
		$\lp$ & $144$\stp \\
		$\lp$ & $122$\stp\stp \\ \hline
		&$12468$
		\end{tabular} 
        \caption{base 10 lunar multiplication}
    \end{subfigure}
    \caption{\small Two representations of multiset addition}\label{fig:multisumset}
\end{figure}

There is a convenient representation for elements of $\mcal{M}^b$ as an array of $b$ sets. Let $f \in \mcal{M}^b$. Let $\vec{A}_f = (A_1, A_2, \ldots, A_b)$, where each coordinate $A_i$ ($1 \leq i \leq b$) is a finite subset of $\nats$ defined as follows: for $a \in \nats$ we have $a \in A_i \iff f(a)\geq i$. The multisumset operation is now defined coordinatewise. Given two such arrays $\vec{A} = (A_1, \ldots, A_b)$ and $\vec{B} = (B_1, \ldots, B_b)$ we define $\vec{A}+\vec{B} = (A_1+B_1, \ldots, A_b+B_b)$, with the convention that $S+\emptyset = \emptyset$ for any $S \subseteq \nats$. This operation makes $\mcal{M}^b$ into a commutative monoid. Note that one of the features of this representation is that the coordinates form a descending chain $A_1 \supseteq A_2 \supseteq \cdots \supseteq A_b$.

For $b \geq 1$, let $\mcal{B}$ denote the set of base-$(b+1)$ sequences. There is a natural bijection $\beta: \mcal{M}^b \to \mcal{B}$. First, $\beta(\emptyset) = 0|_b$. Next, let $f\in \mcal{M}^b_k\setminus \emptyset$. Then,
\[\beta(f) = f(k)f(k-1)\ldots f(1)f(0)|_{b+1}.\]

\begin{thm}
	$\beta: \mcal{M}^b \to \mcal{B}$ is a monoid-homomorphism between $\mcal{M}^b$ equipped with multisumset addition, and base $(b+1)$ numbers equipped with lunar multiplication.
\end{thm}

\begin{proof}
	Let $\mathbf{0}$ denote the multiset $\{0, 0, \ldots, 0\}$ with $b$ repetitions of $0$. Then, $\beta(\mathbf{0}) = b|_{b+1}$, which is the maximal digit. This shows that the neutral element is mapped to the neutral element.

	Next, let $f, g \in \mcal{M}^b$. We need to prove 
	\[\beta(f+g) = \beta(f)\lp\beta(g).\]
	The claim is clear if one of $f, g$ is the emptyset. Assume therefore that $f, g \in \mcal{M}^b\setminus \{\emptyset\}$. Consider the base $b+1$ lunar product $\beta(f)\times \beta(g) = c$:
	\begin{align*}
		&\base{b+1}{f(k)f(k-1)\ldots f(1)f(0)} \lm \base{b+1}{g(m)g(m-1)\ldots g(1)g(0)}\\
		&= \base{b+1}{c_{k+m}c_{k+m-1}\ldots c_1c_0}.
	\end{align*}
	Let $\vec{A} = (A_1, \ldots, A_b)$ be the set-array representation of $\beta(f)$, and $\vec{B} = (B_1, \ldots, B_b)$ be the set-array representation of $\beta(g)$. Let $x_{\vec{A}+\vec{B}}$ be the base-$(b+1)$ representation of the multisumset $\vec{A}+\vec{B} = (A_1+B_1, \ldots, A_b+B_b)$. It is a $(k+m)$-digit number, and for $0 \leq j \leq k+m$, the digit $d_j$ is the number of sets $A_i+B_i$ (for $1 \leq i \leq b$) containing $j$.

	Consider all compositions of $j$ as the sum of two numbers:
	\[0+j, 1+(j-1), 2+(j-2), \ldots, j+0\]
	The sum $k+(j-k)$ appears in $A_1+B_1$ if and only if $a_k \geq 1$ and $b_{j-k} \geq 1$. It appears in $A_2+B_2$ if and only if $a_k\geq 2$ and $b_{j-k}\geq 2$; in general it will appear in exactly $\min(a_k, b_{j-k})$ of the $A_i+B_i$. The number of sets $A_i+B_i$ (for $1 \leq i \leq n$) containing $j$ is therefore,
	\[d_j = \max\{\min(a_0, b_{j}), \min(a_1, b_{j-1}), \ldots, \min(a_j, b_{0})\}\]
	(with the convention that $a_i = 0$ for $i \geq k$, and similarly $b_i = 0$ for $i \geq m$). However, this is exactly
	\[c_j = (a_0\lm b_j) \lp (a_1\lm b_{j-1}) \lp \cdots \lp (a_j\lm b_0).\]
\end{proof}

\begin{defi}
Let $f, g \in \mcal{M}^b$ be two multisets with corresponding set-array representations $\vec{A} = (A_1, \ldots, A_b)$, $\vec{B} = (B_1, \ldots, B_b)$. We say that $g$ is a \demph{divisor} of $f$ (or sometimes, $\vec{B}$ is a divisor of $\vec{A}$) if there exists some multiset $h \in \mcal{M}^b$ with set array representation $\vec{C} = (C_1, \ldots, C_b)$ such that $\vec{A} = \vec{B}+\vec{C}$. If $f \neq 0$ is not the constant $0$-function, we define $d(f)$ (or sometimes, $d(\vec{A})$) to be the number of divisors of $f$.
\end{defi}

\begin{lemma}
\label{manytoone}
	Let $b \geq 1$, let $f \in \mcal{M}^b$ be a multiset, and let $\vec{A} = (A_1, \ldots, A_b)$ be its set-array representation. Suppose that $f \neq 0$ is not the constant $0$-function, so that $A_1 \neq \emptyset$.

	Let $f^* \in \mcal{M}^b$ be a multiset given by the set-array representation $\vec{A}^* = (A_1, \emptyset, \emptyset, \ldots, \emptyset)$. Then $d(f^*) \geq d(f)$, and the inequality is strict if $A_2 \neq \emptyset$.
\end{lemma}

\begin{proof}
	Let $g \in \mcal{M}^b$ be a divisor of $f$, with set-array representation $\vec{B} = (B_1, \ldots, B_b)$. That is, there exists some $h \in \mcal{M}^b$ with set-array representation $\vec{C} = (C_1, \ldots, C_b)$ such that $\vec{A} = \vec{B}+\vec{C}$. Letting $h^* \in \mcal{M}^b$ be the multiset given by the set-array representation $\vec{C}^* = (C_1, \emptyset, \emptyset, \ldots, \emptyset)$ we have $\vec{B}+\vec{C}^* = \vec{A}^*$. Thus, every divisor $\vec{B}$ of $\vec{A}$ is also a divisor of $\vec{A}^*$. That is, $d(f^*) \geq d(f)$.

	If $A_2 \neq \emptyset$ then for any divisor $(B_1, \ldots, B_b)$ of $A$ we must have $B_2 \neq \emptyset$. Thus, $(A_1, \emptyset, \ldots, \emptyset)$ is a divisor of $\vec{A}^*$ that is not a divisor of $\vec{A}$. We therefore have in this case $d(f^*)>d(f)$.
\end{proof}

For $b \geq 1$, let us denote by $[k]_b$ the multiset $f \in \mcal{M}^b$ with set-array representation $([k], \emptyset, \emptyset, \ldots, \emptyset)$. Lemma \ref{manytoone} shows that to prove the maximality of $d([k]_b)$, it suffices to prove its maximality among sets of the form $(A, \emptyset, \emptyset, \ldots, \emptyset)$, rather than arbitrary multisets. The next Theorem\footnote{Theorem 17 from \cite{lunar} shows that $d_b({\underbrace{11\ldots 1}_{k}}|_b) = \sum_m H(n, m)(b-1)^m$. Theorem \ref{propdivbase} generalizes this; the formula for $d_b({\underbrace{11\ldots 1}_{k}}|_b)$ then follows from Corollary \ref{crl:parts}.} shows how to count the number of divisors of such multisets in terms of the number of divisors of $A$. The first corollary then implies that $d([k]_b)$ is maximal if $A$ is a 0-rooted set. The second corollary generalizes Lemma \ref{L15} (and is a more explicit version of Lemma 15 in \cite{lunar}) and sets the stage for proving maximality when $A$ is not 0-rooted.

\begin{thm}
\label{propdivbase}
	Let $b \geq 1$, and let $f \in \mcal{M}^b$ be a nonempty set (that is not a proper multiset). Thus, $f$ has the set-array representation $(A, \emptyset, \emptyset, \ldots, \emptyset)$. Then,
	\[d(f) = \sum_{B \mbox{ \tiny divisor of }A} b^{\crd{B}}.\]
\end{thm}

\begin{proof}
	Let $S$ be a set of cardinality $c \in \nats$, and consider the number of possible chains $S_b \subseteq S_{b-1}\subseteq \cdots \subseteq S_2 \subseteq S_1 = S$.
	The question of which sets in the sequence contain $s \in S$ is answered by a single number $1 \leq k \leq c$, which is the largest number such that $s \in S_k$. Thus, each chain is uniquely identified with a sequence $(s_1, s_2, \ldots, s_c)$ where each $1 \leq s_i \leq b$, and there are $b^c$ such sequences.

	Let $B$ be a divisor of $A$, so that there exists some $C$ with $B+C = A$. Each chain $B_b \subseteq B_{b-1} \subseteq \cdots \subseteq B_1 = B$ gives rise to a divisor of $f$ of the form $(B_1, B_2, \ldots, B_b)$, since
	 \[(B, B_2, \ldots, B_b) + (C, \emptyset, \emptyset, \ldots, \emptyset) = (A, \emptyset, \emptyset, \ldots, \emptyset).\]

	Conversely, if $(B_1, B_2, \ldots, B_b)$ is a divisor of $f$, then $B_b \subseteq B_{b-1}\subseteq \cdots \subseteq B_1$, and $B_1$ is a divisor of $A$.
\end{proof}

\begin{crl}
\label{crl:0bases}
	Let $b \geq 1$, and let $f \in \mcal{M}^b_{\leq k}$ be a nonempty $0$-rooted set, with set-array representation $(A, \emptyset, \emptyset, \ldots, \emptyset)$ (where $A \in \mcal{Z}_{\leq k}$). Then, $d(f) \leq d([k]_b)$, and the inequality is strict for $f \neq [k]_b$.
\end{crl}

\begin{proof}
	The promotion-procedure described in Section \ref{s:0-divisors} either adds elements, or leaves the set as is. Thus, to every divisor $B$ of $A$ there corresponds a divisor $B' \in F(B)$ of $[k]$, with $\crd{B'} \geq \crd{B}$. The claim now follows by the formula in Theorem \ref{propdivbase}, and the observation that there are divisors in $[k]$ that do not arise from promotion (cf.~Lemma \ref{claim:0oddunique} and Lemma \ref{claim:even0unique}).
\end{proof}

\begin{crl}
\label{crl:powers}
	Let $b \geq 1$, and let $f \in \mcal{M}^b$ be a nonempty set with set-array representation $(A, \emptyset, \emptyset, \ldots, \emptyset)$. Let $r \eqdef \min(A)$. Then,
	\[d(f) = (r+1)\sum_{B \mbox{ \tiny divisor of }A-\{r\}} b^{\crd{B}}.\]
\end{crl}

\begin{proof}
	According to the proof of Lemma \ref{L15}, each divisor $B$ of $A-\{r\}$ gives rise to a divisor $B+\{k\}$ of $A$, for $0 \leq k \leq r$. Moreover, all divisors of $A$ are of that form. Since $\crd{B} = \crd(B+\{k\})$, we are done by the formula in Theorem \ref{propdivbase}.
\end{proof}

\begin{thm}
\label{thm:bases}
	Let $b \geq 2$, and let $f \in \mcal{M}^b_{\leq k}$ be a nonempty set with set-array representation $(A, \emptyset, \emptyset, \ldots, \emptyset)$. Then, $d(f) \leq d([k]_b)$, and the inequality is strict for $f \neq [k]_b$.
\end{thm}

\begin{proof}
	Let $r \eqdef \min(A)$. If $r=0$, then the claim reduces to Corollary \ref{crl:0bases}. Otherwise, 
	we have by Corollary \ref{crl:powers}
	\[d(f) = (r+1)\sum_{B \mbox{ \tiny divisor of }A-\{r\}} b^{\crd{B}}.\] 
	Let $n\eqdef \max(A)-r$, so that $A - \{r\}$ is a $0$-rooted set in $\mcal{Z}_{n}$. By Corollary \ref{crl:parts} we have
	\[d(f) = (r+1)\sum_{m=1}^n H(n, m) b^{m}\]
	and on the other hand,
	\[d([k]_b) = \sum_{m=1}^{k}H(n+r, m)b^{m}\geq \sum_{m=1}^{n+r}H(n+r, m)b^{m}.\]
	By Corollary \ref{crl:headstrongbounds} we have
	\[\sum_{m=1}^{n+1}H(n+1, m)b^{m}>2\sum_{m=1}^{n}H(n, m)b^{m}\]
	and by induction, for any $r\geq 1$
	\[\sum_{m=1}^{n+r}H(n+r, m)b^{m}>2^r\sum_{m=1}^{n}H(n, m)b^{m}\geq(r+1)\sum_{m=1}^{n}H(n, m)b^{m}\]
	so that $d([k]_b)>d(f)$.
\end{proof}

We can now prove Conjecture \ref{conjbases}.

\begin{thm}
	Let $b \geq 2$, let $f \in \mcal{M}^b_{\leq k}$ be a multiset that is not the constant $0$ function. Then $d(f) \leq d([k]_b)$, and the inequality is strict if $f \neq [k]_b$. Equivalently, in any base $b \geq 3$, among all $k$-digit numbers $n$, $d_b(n)$ has a unique maximum at $n = (b^k-1)/(b-1) = 111\ldots 1|_b$.
\end{thm}

\begin{proof}
	Let $(A_1, A_2, \ldots, A_b)$ be the set-array representation of $f$. Let $f^* \in \mcal{M}^b$ be a multiset given by the set-array representation $\vec{A}^* = (A_1, \emptyset, \emptyset, \ldots, \emptyset)$. Lemma \ref{manytoone} gives $d(f^*) \geq d(f)$ and the inequality is strict if $A_2 \neq \emptyset$. Theorem \ref{thm:bases} then gives $d(f^*)\leq d([k]_b)$ and the inequality is strict if $f^*\neq [k]_b$. Thus, $d(f) \leq d([k]_b)$ and the inequality is strict if $f \neq [k]_b$.
\end{proof}
\section{Further questions}
\label{s:conclusion}
We have seen that sumset divisors of finite subsets of $\nats$ correspond to binary lunar divisors. The setting of lunar arithmetic naturally inspires number-theoretic questions. This paper investigated divisibility questions for sumsets. In \cite{lunar} Appelgate, LeBrun, and Sloane investigate a whole panoply of number-theoretic constructions for lunar numbers. Do other constructions have natural sumsets-counterparts, and if so may lunar arithmetic shed new insights on sumsets? We single out two important examples. One, sumsets of the form $A+A$ correspond to base-$2$ lunar squares, discussed briefly in \S4 of \cite{lunar}.

Two, irreducible finite subsets correspond to base-$2$ lunar primes, investigated in \S3 of \cite{lunar}. We have mentioned in \S\ref{s:preface} Wirsig's proof \cite{wirsing} that almost all subsets of $\nats$ are asymptotically irreducible. If we restrict our attention to finite subsets only, Applegate, LeBrun, and Sloane make a more precise conjecture:

\setcounter{conj}{9}
\begin{conj}[LeBrun et al.]
	Let $\pi_b(k)$ denote the number of base $b$ lunar primes with $k$ digits. Then,
	\[\pi_b(k) \sim (b-1)^2b^{k-2}.\]
\end{conj}
In particular, this predicts that about half of all subsets of $[k]$ are irreducible.

Theorem \ref{crlodd} undergirds many of the results of this paper; in that the proofs of Theorem \ref{crleven} and Theorem \ref{thm:bases} proceed via reductions to the $0$-rooted case. The load-bearing part of the proof of Theorem \ref{crlodd} is the promotion procedure described in Section \ref{s:0-divisors}. However, this procedure is somewhat unique for the interval $[k]$. For example, 
\[\{0,2\}+\{0,4\} = \{0, 2, 4, 6\}\]
and the attempt to promote these to factors of $\{0, 2, 3, 4, 5, 6\}$ is unsuccessful
\[\{0, 2\} + \{0, 1, 3, 4\} = [6]\]
even though $\{0, 2, 4, 6\} \subseteq \{0, 2, 3, 4, 5, 6\}$. This is the difficulty in proving Conjecture \ref{conjodd2} (Part II) regarding the runner-up to $d([k])$. Is there a way to generalize the promotion procedure to other sets?

Another way of attacking \ref{conjodd2} (Part II) is via direct counting. Applegate, LeBrun, and Sloane construct (Theorem 18 in \cite{lunar}) a generating function by considering a subtle relation with restricted compositions, it is then used to show that $d_2(2^k-3)/d_2(2^k-1) \to 1/5$. Section \ref{s:counting} describes a bijection between divisors of $[k]$ and headstrong compositions, which is expanded upon in Section \ref{s:headstrongtriangle}. Is there a similar bijection between divisors of arbitrary sets and different kind of compositions?

\section{Acknowledgements}
Many thanks to Professor Almut Burchard for her expert guidance and unflagging positivity. I would also like to thank Professor Leo Goldmakher for teaching me about sumsets at the University of Toronto. The question about the number of divisors a sumset may have was first raised at a brainstorming session with Professor Goldmakher and his SMALL team of students Huy Pham, Sophia Dever and Vidya Venkatesh during my visit to Williams College in 2015.

I am also grateful to Brady Haran for his delightful Numberphile videos. It was his interview with Neil Sloane \cite{brady} which prompted this paper.

\newpage

\end{document}